\documentclass[10pt,final,3p,times]{elsarticle}
\usepackage{amsthm} 
\usepackage{lipsum}
\usepackage{amsfonts}
\usepackage{graphicx}
\usepackage{natbib}
\usepackage{hyperref}
\usepackage{epsfig}
\usepackage{amssymb}
\usepackage{cite}
\usepackage{float}
\usepackage{mathtools}
\newtheorem{theorem}{Theorem}[section]
\begin{document}
\begin{frontmatter}
\title{A new method to find exact solution of nonlinear ordinary differential equations: Application to derive thermophoretic waves in graphene sheets   
}
\author{Prakash Kumar Das}
\ead{prakashdas.das1@gmail.com}


\address{Department of Mathematics, TDB College, Raniganj 713347, West Bengal, India}

\begin{abstract}
This article proposes a novel approach for determining exact solutions to nonlinear ordinary differential equations. We integrate a linearization approach for nonlinear equations with the variation of parameters method and Adomian polynomials to obtain exact solutions for higher-order nonlinear boundary and initial value problems. The suggested approach comprises three part. In the initial segment of this approach, a linearization strategy is suggested to transform higher-order nonlinear ordinary differential equations into an infinite system of linear nonhomogeneous ordinary differentials utilizing the Adomian polynomial. In the second phase of this approach, these equations are resolved using the variation of parameters method. In the final segment of this approach, we merge the outcomes from earlier steps to derive a generalized iterative method for addressing the given nonlinear equation. The recommended iterative method provides the solution via a rapidly converging series that readily approaches a closed form solution. The proposed approach is very efficient and essentially perfect for determining exact solutions of nonlinear equations. To demonstrate the effectiveness of this method, we examined the extended (2 + 1) dimensional equation for thermophoretic motion, which is based on wrinkle wave movements in graphene sheets supported by a substrate. The implementation of the suggested approach effectively yielded  closed-form solutions in terms of  exponential functions, hyperbolic functions, trigonometric functions, algebraic functions, and Jacobi elliptic functions, respectively. Three generated solutions illustrated to examine the characteristics of thermophoretic waves in graphene sheets. The proposed method's benefits and drawbacks are also examined. Consequently, unlike previous solutions obtained via the variation of parameters method for nonlinear issues, the solutions presented here are exact and unique. 
\end{abstract}
\end{frontmatter}
\section{Introduction}
The method of variation of parameters (MVP) was developed independently by Leonhard Euler in 1748 and Joseph Louis Lagrange in 1774. The MVP  is a method used to convert solutions of linear homogeneous ordinary differential equations into a specific integral for the related inhomogeneous system. Furthermore, the approach can be considered a highly ingenious enhancement of the order reduction technique used for addressing nonhomogeneous linear differential equations. Variation of parameters for local fractional nonhomogenous linear-differential equations in \citep{horani2016variation} to find articular solution.
Recently, the MVP has proven effective in deriving approximate solutions of nonlinear boundary value problems. This approach was employed for the first time to a micropolar flow problem for deriving its  numerically approximate solutions \citep{gungor2020application}.
The MVP is utilized to iteratively derive numerical solutions for combined conduction and radiation in a one-dimensional, planar, absorbing, emitting, non-gray medium with non-gray opaque boundaries \citep{moore2014application}.
In \citep{noor2008variation,mohyud2009variation}, the authors utilized method of variation of parameter  to derive numerical solutions of nonlinear sixth-order boundary value problems. 

In a recent study \citep{moore2019comparison}, numerical outcomes derived from the MVP are contrasted with those from semi-analytical techniques, specifically the Adomian decomposition method and the differential transformation method, in addressing nonlinear boundary value problems found in engineering. It has been observed that the MVP is considerably more effective than both semi-analytical and conventional numerical approaches,  while achieving similar accuracy. In contrast to semi-analytical techniques, the effectiveness of MVP does not depend on the level of nonlinearity. The MVP is demonstrated to be an appealing substitute for semi-analytical techniques and conventional numerical methods in addressing boundary value problems found in engineering applications where solution efficiency is crucial.

Till now it is believed by the authors that the definite integral involved in the MVP is physically not  usable, is of no practical value or importance while solving nonlinear differential equations. Hence they used  definite integrals which can  accurately approximated for specific values of independent variable using any decent numerical integration method. As a result, it  may not be able to give a nice formula for particular integral, it can still evaluated  at desired points on any reasonable interval of interest, and possibly using these values to generate a table for particular integral or to sketch its graph. For this reason in all of the above cited papers authored have used definite integrals of MVP in solving nonlinear problems, which leads them to approximate solutions.  In construct to the above, here for the first time we are going to use indefinite integral in search of closed form exact solution of nonlinear differential equations. 

From the preceding discussion, it is evident that the  MVP with definite integration has thus far been utilized on nonlinear equations to derive numerical and approximate solutions. Motivated by its precision and achievements in surpassing other techniques, we are compelled to inquire if this method can be adjusted and utilized to obtain exact closed-form solutions of nonlinear differential equations as well? This study seeks to deliver the answer to that inquiry. Therefore, we suggest a modified method of variation of parameters (MMVP) utilizing indefinite integrals to obtain exact solutions of nonlinear ordinary differential equations.
 
The section \ref{MVP} covers the specific procedures for addressing higher order nonlinear differential equations using MMVP. The subsection \ref{red} demonstrates the method for reducing higher order nonlinear differential equations to a system of linear equations. Subsection \label{sMVP} addresses the specific steps of MMVP for addressing the linear system of differential equations presented in previous sections. The final subsection \ref{sSystem} of section \ref{MVP} demonstrates how we can combine the findings from earlier sections to derive a conclusive recursive formula for addressing the specified higher order nonlinear differential equation. Additionally, how to obtain closed-form solutions for the considered equations discussed here. Section \ref{Appl} addresses the use of MMVP to tackle the extended (2 + 1) dimensional thermophoretic motion equation, which emerges from wrinkle wave actions in graphene sheets supported by substrates utilizing the novel approach. We have successfully derived a closed-form solutions involving exponential, hypergeometric, trigonometric, algebraic, and Jacobi elliptic functions for the given equation through MMVP. Section \ref{AdLim} discusses the benefits and drawbacks of the proposed method.  A discussion of our findings concludes Section \ref{con}.

\section{MMVP for higher order nonlinear differential equations}\label{MVP}
\subsection{Reduction of higher order nonlinear differential equations to system of linear equations}\label{red}
Consider the following n-th order nonlinear differential equation:
\begin {eqnarray}\label{eq1p1}
P_0(x)\ y^{(n)}(x) + P_1(x)\ y^{(n-1)}(x) +\cdots+ P_n(x)\ y(x) = {\cal N}[y](x), \, \, \, \, x \in \mathbb{R}.
\end{eqnarray}
 In this case ${\cal N}[y](x)$ is a set of nonlinear terms, where $m$ and $n$ are integers and $P_0(x), P_1(x)\ \cdots, P_n(x)$ are continuous functions of $x$.  
Equation Eq. (\ref{eq1p1}) will be shortened as 
\begin {eqnarray}\label{eq1p2a}
\hat{{\cal L}}[y](x)=  {\cal N}[y](x),
\end{eqnarray}
where
\begin {eqnarray}\label{eq1p2}
\hat{{\cal L}}[y](x)= P_0(x)\ y^{(n)}(x) + P_1(x)\ y^{(n-1)}(x) +\cdots+ P_n(x)\ y(x). 
\end{eqnarray}
To convert Eq. (\ref{eq2p1}) into a linear equation system, we write
\begin{equation}\label{eq1p3}
y(x) = \sum_{k=0}^\infty \epsilon^{\frac{k}{2}} \, y_{k}(x)
\end{equation}  
and reformulate the nonlinear term into the series 
\begin{equation}\label{eq1p4}
{\cal N}[y](x) = \epsilon^{p} \, \sum_{k=0}^\infty \epsilon^{\frac{k}{2}} {\cal A}_k(y_{0}(x),y_{1}(x),\cdots ,y_{k}(x))\equiv \epsilon^{p} \, \sum_{k=0}^\infty \epsilon^{\frac{k}{2}} \, {\cal A}_k(x),
\end{equation}
where ${\cal A}_k(x) \ (\ k\geq 0) $ are referred to as Adomian polynomials \citep{adomian2013solving,das2018solutions}. These polynomials can be derived methodically by applying the formula
\begin {equation}\label{eq1p5}
{\cal A}_k(x)= \frac{1}{k!}\left[\frac{d^k}{d \bar{\epsilon}^k}{\cal N} \left( \sum_{l=0}^\infty y_l(x) \ \bar{\epsilon}^l \right)\right]_{\bar{\epsilon}=0}, \ \ k \geq 0.
\end{equation}
It should be noted that the symbols $\epsilon$ and $\bar{\epsilon}$ found in (\ref{eq1p3})-(\ref{eq1p5}) do not represent any physical perturbation parameter. Rather, these serve as the placeholders for nonlinear terms during subsequent iterations.
Utilizing (\ref{eq1p3}) and (\ref{eq1p4}) in (\ref{eq1p2a}), we obtain
\begin {eqnarray}\label{eq1p6}
 \sum_{k=0}^\infty \epsilon^{\frac{k}{2}} \, \hat{\cal L}\left[y_{k}(x) \right] =\epsilon^{p} \, \sum_{k=0}^\infty \epsilon^{\frac{k}{2}} \, {\cal A}_k(x).
\end{eqnarray}
It is noteworthy that gathering the coefficients of $\epsilon$ after substituting $p = \frac{1}{2}$ in (\ref{eq1p6}) results in a system of linear equations 
\begin{equation}\label{eq1p7}
  \hat{\cal L}\left[y_{0}(x) \right]=0, \ \ 
  \hat{\cal L}\left[y_{k}(x) \right] = {\cal A}_{k-1}(x),\ \ k \geq 1
\end{equation}
involving dependent variables $y_{k}(\xi)$ and non-homogeneous terms ${\cal A}_{k-1}(x)$.
Conversely, selecting $p=1$ in (\ref{eq1p6}) provides the linear system
\begin{equation}\label{eq1p8}
  \hat{\cal L}\left[y_{0}(x) \right]=0, \ \
  \hat{\cal L}\left[y_{1}(x) \right]=0, \ \
   \hat{\cal L}\left[y_{k}(x) \right] = {\cal A}_{k-2}(x),\ \ k \geq 2.
\end{equation}
In the next section, we will now utilize the modfied method of variation of parameters for the system of equations (\ref{eq1p7}) and (\ref{eq1p8}).
\subsection{Solution of system of linear differential equations (\ref{eq1p7}) and (\ref{eq1p8}) by MMVP}\label{sMVP}
In this part, we address the nonhomogeneous equation
\begin {eqnarray}\label{eq2p1}
\hat{\cal L}\left[y_{k}(x) \right] = F_k(x)
\end{eqnarray}
and  use $F_k(x)={\cal A}_{k-1}(x)$ when addressing system (\ref{eq1p7}) and $F_k(x)={\cal A}_{k-2}(x)$ for system (\ref{eq1p8}). 
It is crucial to mention that all solutions to this equation and its complementary equation $\hat{\cal L}\left[y_{k}(x) \right] = 0$ are taken into account in $ \mathbb{R}$. Now the complementary solution of linear part of all equations i.e., $\hat{\cal L}\left[y_{k}(x) \right] =  0,\ i=0, 1 ,2, \cdots $ has the following form
\begin {eqnarray}\label{eq2p2}
y_{k}^{c}(x) = u_1 \  \tilde{y}_{1}(x) + u_2 \  \tilde{y}_{2}(x) +\cdots+ u_n \  \tilde{y}_{n}(x),
\end{eqnarray}
where $u_i,\ i=1,2,\cdots n$ arbitrary constants and  set of solutions $\{  \tilde{y}_{1}(x),  \tilde{y}_{2}(x), \cdots ,  \tilde{y}_{n}(x)\}$ are known as  fundamental set of solutions. 
We will demonstrate how to apply the MMVP to determine a specific solution of (\ref{eq2p1}), when we know a fundamental set of solutions $\{  \tilde{y}_{1}(x),  \tilde{y}_{2}(x), \cdots ,  \tilde{y}_{n}(x)\}$.
We are looking for a particular solution of (\ref{eq2p1}) in the format of 
\begin {eqnarray}\label{eq2p3}
 y_{k}^{p}(x) = u_1(x) \   \tilde{y}_{1}(x) + u_2(x)\   \tilde{y}_{2}(x) +\cdots+ u_n(x)\   \tilde{y}_{n}(x)
\end{eqnarray}
where $\{ u_1(x), u_2(x), \cdots , u_n(x)\}$ are functions yet to be identified. We start by applying the subsequent $(n-1)$ constraints on $\{ u_1(x), u_2(x), \cdots, u_n(x)\}$: 
\begin {eqnarray}\label{eq2p5}
\begin{cases}
u_1'(x) \   \tilde{y}_{1}(x) + u_2'(x) \   \tilde{y}_{2}(x) +\cdots+ u_n'(x) \   \tilde{y}_{n}(x)= 0,\\
u_1'(x) \   \tilde{y}_{1}'(x) + u_2'(x) \   \tilde{y}_{2}'(x) +\cdots+ u_n'(x) \   \tilde{y}_{n}'(x)= 0, \\
      \hspace{1in} \vdots  \\
u_1'(x) \   \tilde{y}_{1}^{(n-2)}(x) + u_2'(x) \   \tilde{y}_{2}^{(n-2)}(x) +\cdots+ u_n'(x) \   \tilde{y}_{n}^{(n-2)}(x)= 0. 
\end{cases}
\end{eqnarray}
These circumstances result in straightforward equations for the higher derivatives of $y_k(x)$:
\begin {eqnarray}\label{eq2p6}
y_k^{(r)}(x) = u_1(x) \   \tilde{y}_{1}^{(r)}(x) + u_2(x)\   \tilde{y}_{2}^{(r)}(x) +\cdots+ u_n(x)\   \tilde{y}_{n}^{(r)}(x), \  \ 0 \leq r \leq (n-1).
\end{eqnarray}
The final equation in Equation (\ref{eq2p6}) is 
\begin {eqnarray}\label{eq2p6a}
y_k^{(n-1)}(x) = u_1(x) \   \tilde{y}_{1}^{(n-1)}(x) + u_2(x)\   \tilde{y}_{2}^{(n-1)}(x) +\cdots+ u_n(x)\   \tilde{y}_{n}^{(n-1)}(x).
\end{eqnarray}
Taking the derivative of this results in 
\begin {eqnarray}\label{eq2p7}
y_k^{(n)}(x) = u_1(x) \   \tilde{y}_{1}^{(n)}(x) + u_2(x)\   \tilde{y}_{2}^{(n)}(x) +\cdots+ u_n(x)\   \tilde{y}_{n}^{(n)}(x)+
u_1'(x)\   \tilde{y}_{1}^{(n-1)}(x) + u_2'(x)\   \tilde{y}_{2}^{(n-1)}(x) +\cdots+ u_n'(x)\   \tilde{y}_{n}^{(n-1)}(x).
\end{eqnarray}
Inserting this and equation (\ref{eq2p6}) into equation (\ref{eq2p1}) produces
\begin {eqnarray}\label{eq2p8}
 && u_1(x) \ \hat{\cal L}\left[  \tilde{y}_{1}\right](x) + u_2(x)\ \hat{\cal L}\left[  \tilde{y}_{2}\right](x) +\cdots+ u_n(x)\ \hat{\cal L}\left[  \tilde{y}_{n}\right](x) \nonumber \\
 && +  P_0(x)\ \left( u_1'(x) \   \tilde{y}_{1}^{(n-1)}(x) + u_2'(x)\   \tilde{y}_{2}^{(n-1)}(x) +\cdots+ u_n'(x)\   \tilde{y}_{n}^{(n-1)}(x) \right) = F_k(x).
\end{eqnarray}
Given that $\hat{\cal L}\left[  \tilde{y}_{i}\right](x) = 0, \ (1\leq i \leq n)$, this simplifies to
\begin {eqnarray}\label{eq2p9}
  \left( u_1'(x) \  \tilde{y}_1^{(n-1)}(x) + u_2'(x)\  \tilde{y}_2^{(n-1)}(x) +\cdots+ u_n'(x)\  \tilde{y}_n^{(n-1)}(x) \right) =\frac{ F_k(x)}{P_0(x)}.
\end{eqnarray}
 It indicates that (\ref{eq2p3}) is a solution to equation (\ref{eq2p1}) if conditions (\ref{eq2p5}) and (\ref{eq2p9}) are satisfied. It is essential to mention that the combination of Equations (\ref{eq2p5}) and (\ref{eq2p9}) can be expressed in matrix form as
\begin {eqnarray}\label{eq2p10}
 \left[\begin{array}{cccc}
   \tilde{y}_{1}(x)&   \tilde{y}_{2}(x)&\cdots &   \tilde{y}_{n}(x) \\   \tilde{y}_{1}'(x)&   \tilde{y}_{2}'(x)&\cdots &   \tilde{y}_{n}'(x)   \\ \vdots & \vdots &\ddots & \vdots \\   \tilde{y}_{1}^{(n-2)}(x)&   \tilde{y}_{2}^{(n-2)}(x)&\cdots &   \tilde{y}_{n}^{(n-2)}(x)  \\    \tilde{y}_{1}^{(n-1)}(x)&   \tilde{y}_{2}^{(n-1)}(x)&\cdots &   \tilde{y}_{n}^{(n-1)}(x)  \\\end{array}\right]\left[\begin{array}{c} u_1'(x)\\ u_2'(x)\\ \vdots \\ u_{n-1}'(x)  \\ u_n'(x) \\ \end{array}\right]=\left[\begin{array}{c}
 0\\ 0\\ \vdots \\ 0  \\ F_k(x) /P_0(x) \\ \end{array}\right].
 \end{eqnarray}
 Applying Cramer’s rule to solve the system of equations results in 
\begin {eqnarray}\label{eq2p11}
u_{j}'(x) = (-1)^{n-j} \frac{F_k(x)\ W_j(x)}{P_0(x)\ W(x)}, 1 \leq j \leq n,  
\end{eqnarray}
where $W(x)$ represents the Wronskian of the fundamental solution set $\{   \tilde{y}_{1}(x),   \tilde{y}_{2}(x), \cdots ,   \tilde{y}_{n}(x) \}$, which is non-zero on $ \mathbb{R}$ and  $W_j(x)$ is the determinant derived by removing the last row and the $j-$th column from $W(x)$.

Once we have acquired $ u_1'(x), u_2'(x), \cdots , u_n'(x)$, we can perform integration to derive $ u_1(x), u_2(x), \cdots , u_n(x)$. We consider the integration constants as zero and eliminate any linear combination of $\{ \tilde{y}_{1}(x),   \tilde{y}_{2}(x), \cdots ,   \tilde{y}_{n}(x) \}$ that could be present in $y_k(x)$. This provides the solution to equation (\ref{eq2p1}) in the format 
\begin {eqnarray}\label{eq2p12}
y_k(x) = \sum_{j=1}^{n}(-1)^{n-j}    \tilde{y}_{j}(x) \int \frac{F_k(x) \ W_j(x)}{P_0(x)\ W(x)} dx.   
\end{eqnarray}
\subsection{Solution of system of equations (\ref{eq1p7}) and (\ref{eq1p8})}\label{sSystem}
To address the system (\ref{eq1p7}) through MMVP, we need to solve the recursive system
\begin {eqnarray}\label{eq2p13}
\begin{cases}
y_0(x) = \sum_{j=1}^{n} c_{0,j} \  \tilde{y}_{j}(x), \nonumber \\
y_k(x) = \sum_{j=1}^{n}(-1)^{n-j}    \tilde{y}_{j}(x) \int \frac{{\cal A}_{k-1}(x) \ W_j(x)}{P_0(x)\ W(x)} dx,   \ \ k \geq 1,
\end{cases}
\end{eqnarray}
where $c_{0,j}, \ j=1,2, \cdots n,$ are chosen constants and $\tilde{y}_{j}(x)$ are the solutions to $\hat{\cal L}\left[y_{k}(x) \right] = 0.$ This results in the series solution of equation (\ref{eq1p1}) expressed in the form
\begin{equation}\label{eq2p14}
y(x) = \sum_{k=0}^\infty  y_{k}(x)
\end{equation} 
In the same way, system (\ref{eq1p8}) can yield a series solution in the format of (\ref{eq2p14}) by utilizing a recursive system
\begin {eqnarray}\label{eq2p15}
\begin{cases}
 y_0(x) = \sum_{j=1}^{n} c_{0,j} \  \tilde{y}_{j}(x), \\
y_1(x) = \sum_{j=1}^{n} c_{1,j} \  \tilde{y}_{j}(x), \\ 
y_k(x) = \sum_{j=1}^{n}(-1)^{n-j}    \tilde{y}_{j}(x) \int \frac{{\cal A}_{k-2}(x) \ W_j(x)}{P_0(x)\ W(x)} dx, \ k \geq 2, 
\end{cases}
\end{eqnarray}
where $c_{l,j}, \ l=0,1$ are arbitrary constants and $ \tilde{y}_{j}(x)$ are solutions to $\hat{\cal L}\left[y_{k}(x)\right] = 0$ for  $j=0,1,2,\cdots,n.$

It is essential to point out that we have set the integration constants to zero during the integration of \( u_j'(x), \ j=1, 2, \cdots, n \), in (\ref{eq2p12}). If we retain them, they will be absorbed into the terms of \( y_0(x) \) of the series solution (\ref{eq2p14}) without yielding any significantly new results. It is essential to recognize that utilizing symbolic computations often enables us to accurately determine the general term or generating functions of the iterative scheme (\ref{eq2p13}) or (\ref{eq2p15}). Whenever we can determine the general term or generating functions, the series sum (\ref{eq2p14}) consistently provides a closed-form solution to the nonlinear differential equation (\ref{eq1p1}).
\section{Application of MMVP}\label{Appl}
\textbf{\textit{Example 1:}} Consider the following extended to (2 + 1) dimensional  thermophoretic motion equation, which was derived from wrinkle wave motions in substrate-supported graphene sheets\citep{abdel2021multi}
\begin{eqnarray}\label{ex2eq1}
u_{tx}(x,y,t)+(u_x(x,y,t))^2 +u(x,y,t)\ u_{xx}(x,y,t)+u_{xxxx}(x,y,t)+\left(\alpha(t) +b\right) u_{xx}(x,y,t)+\delta(t)\ u_{yy}(x,y,t)=0, 
\end{eqnarray}
where $u(x,y,t)$ is thermophoretic moving variable respect to longitudinal
displacement $x$, the lateral displacement $y$ and time $t$, $\alpha(t)$ and $b$ are the thermal conductivity coefficients, and $\delta(t)$ is the coefficient of the
lateral dispersion. It can be considered as the driving force for the accelerated movement of wrinkles from the thermal slope.
Under the following transformation 
\begin{eqnarray}\label{ex2eq2}
u(x,y,t)=U(\xi), \ \ \xi= -\frac{1}{\sqrt{a_1}} \int \left\{a_1 \left( \alpha (t) + b\right)+c_2^2 \ \delta (t)-a_2\right\} \, dt+\sqrt{a_1} \ x+c_2 \ y-c_1,
\end{eqnarray}
 equation (\ref{ex2eq1}) reduces to the following ordinary differential equation
\begin {eqnarray}\label{ex2eq3}
a_1^2 \ U^{(4)}(\xi )+ a_2 \ U^{(2)}(\xi ) =- a_1 \left\{ \left(U'(\xi )\right)^2+   U(\xi )\ U^{(2)}(\xi )\right\}.
\end{eqnarray}
If we compair it with equation (\ref{eq1p1}), we get $n=4,\ P_0(\xi)= a_1^2,\ P_1(\xi)= 0, \ P_2(\xi)= a_2, \ P_{3,4}(\xi)= 0$ and ${\cal N}[U](\xi)=- a_1 \left\{ \left(U'(\xi )\right)^2+   U(\xi )\ U^{(2)}(\xi )\right\}$.
Next following the section \ref{red}, we can reduce the equation (\ref{ex2eq3}) to the linear system of equation  (\ref{eq1p7}). Here the fundamental set of solutions of the linear part $\hat{\cal L}\left[U(\xi) \right]  \simeq a_1^2 \ U^{(4)}(\xi )+ a_2 \ U^{(2)}(\xi )= 0$ is given by 
 $\left\{ \tilde{U}_1(\xi)= e^{-\frac{\sqrt{-a_2} \xi }{a_1}},\ \tilde{U}_2(\xi)=e^{\frac{\sqrt{-a_2} \xi }{a_1}}, \ \tilde{U}_3(\xi)=\xi ,\ \tilde{U}_4(\xi)=1\right\}$ 
  This system of equations can be solved by MMVP using the iterative scheme given by (\ref{eq2p13}) in the form
 \begin {eqnarray}\label{ex2eq4}
 \begin{cases}
  U_0(\xi)=c_{0,1}\  e^{-\frac{\sqrt{-a_2} \xi }{a_1}}+c_{0,2}\  e^{\frac{\sqrt{-a_2} \xi }{a_1}}+c_{0,3}\ \xi +c_{0,4} \left(\equiv \sum_{j=1}^{n} c_{0,j} \  \tilde{U}_{j}(\xi)\right),  \\
 U_k(\xi) = \sum_{j=1}^{n}(-1)^{n-j}    \tilde{U}_{j}(\xi) \int \frac{{\cal A}_{k-1}(\xi) \ W_j(\xi)}{P_0(\xi)\ W(\xi)} d\xi,   \ \ k \geq 1,
\end{cases}
\end{eqnarray}
where ${\cal A}_{k-1}(\xi)$ are Adomian polynomials of nonlinear term  ${\cal N}[U](\xi)$ calculated using the formula (\ref{eq1p5}),  $W(\xi)$ represents the Wronskian of the fundamental solution set $\left\{ \tilde{U}_1(\xi),\ \tilde{U}_2(\xi), \ \tilde{U}_3(\xi),\ \tilde{U}_4(\xi)\right\}$,  and  $W_j(\xi)$ is the determinant derived by removing the last row and the $j-$th column from $W(\xi)$.\\ 
\textbf{Case-1.\  For Boundary Condition $U(\infty)= 0.$ }\\
Now applying the boundary condition $U(\xi)\rightarrow 0$ as $\xi \rightarrow \infty $ for localized solution on recursive scheme (\ref{ex2eq4}), we get $c_{0,i}=0,\ i=2,3,4$ which gives $U_0(\xi)=c_{0,1}\  e^{-\frac{\sqrt{-a_2} \xi }{a_1}}$ when $a_2<0$ and $a_1>0$. With this updated $U_0(\xi)$ iterative scheme (\ref{ex2eq4}) gives us the following correction terms
\begin {eqnarray}\label{ex2eq5}
&& U_0(\xi)=c_{0,1}\  e^{-\frac{\sqrt{-a_2} \xi }{a_1}},\ U_1(\xi)=\frac{a_1\  c_{0,1}^2 e^{-\frac{2 \sqrt{-a_2} \xi }{a_1}}}{6\ a_2},\ U_2(\xi)=\frac{a_1^2\ c_{0,1}^3
   e^{-\frac{3 \sqrt{-a_2} \xi }{a_1}}}{48\ a_2^2},\\ 
   && U_3(\xi)=\frac{a_1^3 \ c_{0,1}^4 \ e^{-\frac{4 \sqrt{-a_2} \xi }{a_1}}}{432\ a_2^3}, \ U_4(\xi)=\frac{5\ a_1^4 \ c_{0,1}^5
   e^{-\frac{5 \ \sqrt{-a_2} \xi }{a_1}}}{20736\  a_2^4},\ U_5(\xi)=\frac{a_1^5\ c_{0,1}^6 e^{-\frac{6 \sqrt{-a_2} \xi }{a_1}}}{41472\ a_2^5}, \cdots 
\end{eqnarray}
Now using the symbolic computation, one can get the general term of the above iterative terms in the form
\begin {eqnarray}\label{ex2eq6}
&& U_n(\xi)=\frac{(n+1)\ a_2  }{12^{n} \ a_1}\left(\frac{a_1 \ c_{0,1}}{a_2}  e^{-\frac{\sqrt{-a_2} }{a_1}\xi }\right)^{n+1}, \ n\geq 0. 
\end{eqnarray}
Hence the sum  of the above iterative terms in closed form is  given by the generating function
\begin{equation}\label{ex2eq7}
U(\xi) = \sum_{n=0}^\infty  U_{n}(\xi)= \frac{144 a_2^2 c_{0,1} e^{\frac{\sqrt{-a_2} \xi }{a_1}}}{\left(12 a_2 e^{\frac{\sqrt{-a_2} \xi }{a_1}}-a_1 c_{0,1}\right)^2}.
\end{equation} 
It can be checked that it satisfies equation (\ref{ex2eq3}), so it is exact solution. Now combination of  (\ref{ex2eq7}) and  (\ref{ex2eq2}) gives exact solution of  (\ref{ex2eq1}) in the form 
\begin{equation}\label{ex2p8}
u(x,y,t) =\frac{144 \ a_2^2\ c_{0,1} \ \exp \left(\frac{\sqrt{-a_2} }{a_1}\left(-\frac{1}{\sqrt{a_1}} \int \left\{a_1 \left( \alpha (t) + b\right)+c_2^2 \ \delta (t)-a_2\right\} \, dt+\sqrt{a_1} \ x+c_2 \ y-c_1\right)\right)}{\left[-a_1\ c_{0,1}+12\ a_2\  \exp \left(\frac{\sqrt{-a_2}}{a_1} \left(-\frac{1}{\sqrt{a_1}} \int \left\{a_1 \left( \alpha (t) + b\right)+c_2^2 \ \delta (t)-a_2\right\} \, dt+\sqrt{a_1} \ x+c_2 \ y-c_1\right)\right)\right]^2}.
\end{equation} 
Note that above solution is real and bounded for $a_1>0,\ a_2<0, \ c_{0,1}>0.$ In the figure \ref{fig1} profiles of solution (\ref{ex2p8}) are displayed for arbitrary function and parameter values (a) $\alpha (t) =-5 \cos (0.6 t+12),\ \delta (t)= 2 \text{sech}(0.1 t)-0.1,\ y =1,\ a_1 = 0.3,\ a_2 = -0.04,\ c_1 = 0.5,\ c_2 = 0.3,\ c_{0,1} = 4,\ b =0.8$ for fixing $y =1$ in the first column and (b) $\alpha (t) =-5 \cos (5 t+12),\ \delta (t)= -5 \text{sech}(0.4 t)+2.5\ t+5, \ a_1 = 0.8,\ a_2 = -0.5,\ c_1 = 0.5,\ c_2 = 0.3,\ c_{0,1} = 8,\ b =0.8$  for fixing $x =1$ in the second column. It is clear from the figure that solution have snake like and periodic boomerang  like profile due to the presence of periodic values of $\alpha (t) ,\ \delta (t)$ in the solution. 
\begin{figure}
\includegraphics[scale=1.25]{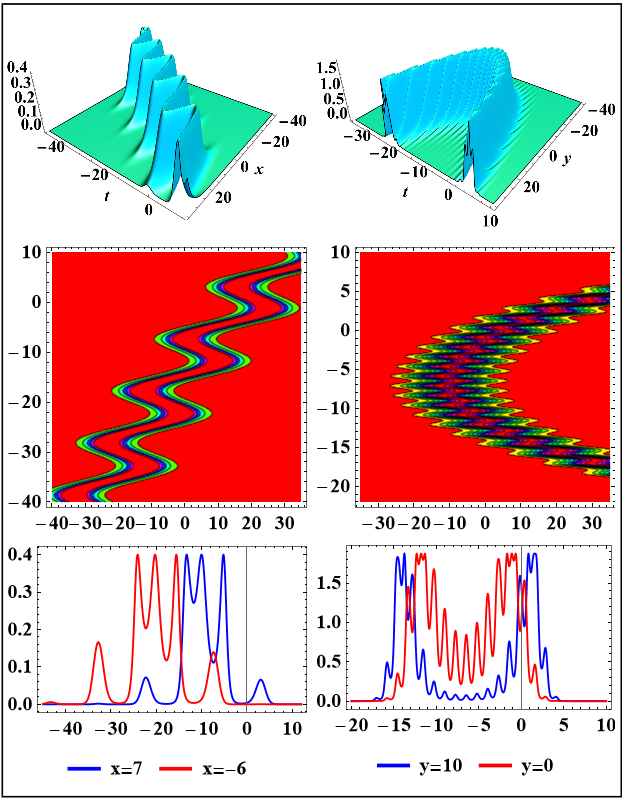}
\caption{3D, Contour and 2D Plots of solution (\ref{ex2p8}) for arbitrary functions and parameter values (a) $\alpha (t) =-5 \cos (0.6 t+12),\ \delta (t)= 2 \text{sech}(0.1 t)-0.1,\ y =1,\ a_1 = 0.3,\ a_2 = -0.04,\ c_1 = 0.5,\ c_2 = 0.3,\ c_{0,1} = 4,\ b =0.8$ in shown first column and for (b) $\alpha (t) =-5 \cos (5 t+12),\ \delta (t)= -5 \text{sech}(0.4 t)+2.5\ t+5,\ x =1, \ a_1 = 0.8,\ a_2 = -0.5,\ c_1 = 0.5,\ c_2 = 0.3,\ c_{0,1} = 8,\ b =0.8$ those are showed in the 2nd column. }\label{fig1}
\end{figure}
It is important to note that based on parameter restrictions, solution (\ref{ex2p8}) can be recast in several hyperbolic and trigonometric form \citep{das2020chirped}. In the following we have shown a few simple  hyperbolic and trigonometric forms of (\ref{ex2p8}). 
\begin{equation}\label{ex2p9}
u(x,y,t) =-\frac{3 a_2}{a_1} \text{sech}^2\left[\frac{\sqrt{-a_2} }{2 a_1}\left(-\frac{1}{\sqrt{a_1}} \int \left\{a_1 \left( \alpha (t) + b\right)+c_2^2 \ \delta (t)-a_2\right\} \, dt+\sqrt{a_1} \ x+c_2 \ y-c_1\right)+\frac{1}{2} \log \left(-\frac{12 a_2}{a_1 c_{0,1}}\right)\right],
\end{equation} 
\begin{equation}\label{ex2p10}
u(x,y,t) =\frac{3 a_2 }{a_1} \text{csch}^2\left[\frac{\sqrt{-a_2} }{2 a_1}\left(-\frac{1}{\sqrt{a_1}} \int \left\{a_1 \left( \alpha (t) + b\right)+c_2^2 \ \delta (t)-a_2\right\} \, dt+\sqrt{a_1} \ x+c_2 \ y-c_1\right)+\frac{1}{2} \log \left(\frac{12 a_2}{a_1 c_{0,1}}\right)\right].
\end{equation} 
The above two hyperbolic form are real and bounded for $a_1>0,\ a_2<0, \ c_{0,1}>0$. Let us consider the following trigonometric form
\begin{equation}\label{ex2p11}
u(x,y,t) =-\frac{3 a_2 }{a_1} \csc ^2\left[\frac{\sqrt{a_2} }{2 a_1} \left(-\frac{1}{\sqrt{a_1}} \int \left\{a_1 \left( \alpha (t) + b\right)+c_2^2 \ \delta (t)-a_2\right\} \, dt+\sqrt{a_1} \ x+c_2 \ y-c_1\right)-\frac{i}{2}  \log \left(\frac{12 a_2}{a_1 c_{0,1}}\right)\right]
\end{equation} 
The above  trigonometric forms is complex, but can be made real for appropriate choice of values of free parameters. For example solution (\ref{ex2p11}) is real for $c_{0,1}=\frac{12 a_2}{a_1 }$ or $c_{1}=-\frac{i a_1}{\sqrt{a_2} }  \log \left(\frac{12 a_2}{a_1 c_{0,1}}\right)$   with  $a_1>0,\ a_2>0$. Also it is important to note that for the above choice of  $c_{0,1}$ and  in case $a_2=0$ this solution reduces to  the algebraic form
\begin{equation}\label{ex2p11a}
u(x,y,t) =-\frac{12 a_1 }{ \left[-\frac{1}{\sqrt{a_1}} \int \left\{a_1 \left( \alpha (t) + b\right)+c_2^2 \ \delta (t)\right\} \, dt+\sqrt{a_1} \ x+c_2 \ y-c_1\right]^2}.
\end{equation} 
 Another simple trigonometric form of (\ref{ex2p8}) is given by
\begin{equation}\label{ex2p12}
u(x,y,t) =-\frac{3 a_2 }{a_1}\sec ^2\left[\frac{\sqrt{a_2} }{2 a_1} \left(-\frac{1}{\sqrt{a_1}} \int \left\{a_1 \left( \alpha (t) + b\right)+c_2^2 \ \delta (t)-a_2\right\} \, dt+\sqrt{a_1} \ x+c_2 \ y-c_1\right)-\frac{i}{2}  \log \left(-\frac{12 a_2}{a_1 c_{0,1}}\right)\right],
\end{equation} 
The above  trigonometric form is also complex, but can be made real for appropriate choice of values of free parameters. For example solution (\ref{ex2p12}) is real for $c_{0,1}=-\frac{12 a_2}{a_1 }$ or $c_{1}=-\frac{i a_1}{\sqrt{a_2} }  \log \left(-\frac{12 a_2}{a_1 c_{0,1}}\right)$   with  $a_1>0,\ a_2>0$.\\
\textbf{Case-2.\  For Boundary Condition $U(-\infty)= 0.$ }\\
Now applying the boundary condition $U(\xi)\rightarrow 0$ as $\xi \rightarrow - \infty $ for localized solution on recursive scheme (\ref{ex2eq4}), we get $c_{0,i}=0,\ i=1,3,4$ which gives $U_0(\xi)=c_{0,2}\  e^{\frac{\sqrt{-a_2} \xi }{a_1}}$ when $a_2<0$ and $a_1>0$. With this updated $U_0(\xi)$ iterative scheme (\ref{ex2eq4}) gives us the following correction terms
\begin {eqnarray}\label{ex2eq13}
&& U_0(\xi)=c_{0,2}\  e^{\frac{\sqrt{-a_2} \xi }{a_1}},\ U_1(\xi)=\frac{a_1\  c_{0,2}^2 e^{\frac{2 \sqrt{-a_2} \xi }{a_1}}}{6\ a_2},\ U_2(\xi)=\frac{a_1^2\ c_{0,2}^3
   e^{\frac{3 \sqrt{-a_2} \xi }{a_1}}}{48\ a_2^2},\\ 
   && U_3(\xi)=\frac{a_1^3 \ c_{0,2}^4 \ e^{\frac{4 \sqrt{-a_2} \xi }{a_1}}}{432\ a_2^3}, \ U_4(\xi)=\frac{5\ a_1^4 \ c_{0,2}^5
   e^{\frac{5 \ \sqrt{-a_2} \xi }{a_1}}}{20736\  a_2^4},\ U_5(\xi)=\frac{a_1^5\ c_{0,2}^6 e^{\frac{6 \sqrt{-a_2} \xi }{a_1}}}{41472\ a_2^5}, \cdots 
\end{eqnarray}
The general term of the above iterative terms is given by
\begin {eqnarray}\label{ex2eq14}
&& U_n(\xi)=\frac{(n+1)\ a_2  }{12^{n} \ a_1}\left(\frac{a_1 \ c_{0,2}}{a_2}  e^{\frac{\sqrt{-a_2} }{a_1}\xi }\right)^{n+1}, \ n\geq 0. 
\end{eqnarray}
Using symbolic calculation the sum  of the above iterative terms in closed form is  given by the generating function
\begin{equation}\label{ex2eq15}
U(\xi) = \sum_{n=0}^\infty  U_{n}(\xi)= \frac{144 a_2^2 c_{0,2} e^{-\frac{\sqrt{-a_2} \xi }{a_1}}}{\left(12 a_2 e^{-\frac{\sqrt{-a_2} \xi }{a_1}}-a_1 c_{0,2}\right)^2}.
\end{equation} 
It can be checked that it satisfies equation (\ref{ex2eq3}), so it is exact solution. Now combination of  (\ref{ex2eq15}) and  (\ref{ex2eq2}) gives exact solution of  (\ref{ex2eq1}) in the form 
\begin{equation}\label{ex2p16}
u(x,y,t) =\frac{144 \ a_2^2\ c_{0,2} \ \exp \left(-\frac{\sqrt{-a_2} }{a_1}\left(-\frac{1}{\sqrt{a_1}} \int \left\{a_1 \left( \alpha (t) + b\right)+c_2^2 \ \delta (t)-a_2\right\} \, dt+\sqrt{a_1} \ x+c_2 \ y-c_1\right)\right)}{\left[-a_1\ c_{0,2}+12\ a_2\  \exp \left(-\frac{\sqrt{-a_2}}{a_1} \left(-\frac{1}{\sqrt{a_1}} \int \left\{a_1 \left( \alpha (t) + b\right)+c_2^2 \ \delta (t)-a_2\right\} \, dt+\sqrt{a_1} \ x+c_2 \ y-c_1\right)\right)\right]^2}.
\end{equation} 
Note that above solution is real and bounded for $a_1>0,\ a_2<0, \ c_{0,2}>0.$  If we compare solution (\ref{ex2p16})   with (\ref{ex2p8}), it can be found that both are same except one sign before the independent variable and alter in the integration constant. So the solution (\ref{ex2p16})  have same features with  solution (\ref{ex2p8}) as shown in the figure \ref{fig1}. Following the previous case solution (\ref{ex2p16}) can also be expressed the forms of solutions (\ref{ex2p9}) to (\ref{ex2p12}), for simplicity we have skipped it here. It is important to note that the above derived solutions (\ref{ex2eq7}) and (\ref{ex2eq15}) of (\ref{ex2eq3}) satisfies the boundary conditions $U(\pm \infty)=0$. One can utilize these solutions to obtain solutions satisfying boundary condition  $\bar{U}(\pm \infty)=\Lambda $ using following theorem:
\begin{theorem}\label{th1}
If $U(\xi,a_1,a_2)$ is a solution of (\ref{ex2eq3}) then $\Lambda + U(\xi,a_1,\Lambda \ a_1 + a_2)$ is also a solution of  (\ref{ex2eq3}), where $\Lambda $ is an arbitrary constant and $a_1$ and $a_2$ are parameters involved in the equation.
\end{theorem}
\begin{proof}
Since $U(\xi)\simeq U(\xi,a_1,a_2)$ is a solution of (\ref{ex2eq3}), it satisfies 
\begin{eqnarray}\label{ex2eq17}
a_1^2 \ U^{(4)}(\xi,a_1,a_2)+ a_2 \ U^{(2)}(\xi,a_1,a_2) + a_1 \left\{ \left(U'(\xi,a_1,a_2))\right)^2+   U(\xi,a_1,a_2)\ U^{(2)}(\xi,a_1,a_2)\right\}=0.
\end{eqnarray}
Next we consider another solution of (\ref{ex2eq3}) in the form 
\begin{eqnarray}\label{ex2eq17a}
\bar{U}(\xi) = \Lambda +V(\xi).
\end{eqnarray}
 Using this solution in (\ref{ex2eq3}), to get
\begin{eqnarray}\label{ex2eq18}
a_1^2 \ V^{(4)}(\xi)+ (a_1 \Lambda + a_2 )\ V^{(2)}(\xi,a_1,a_2) + a_1 \left\{ \left(V'(\xi))\right)^2+   V(\xi)\ V^{(2)}(\xi)\right\}=0.
\end{eqnarray}
If we compare (\ref{ex2eq18}) with (\ref{ex2eq17}), it is clear that  $V(\xi)\simeq U(\xi,a_1,a_1 \Lambda + a_2 )$ is a solution of (\ref{ex2eq18}). That is $\bar{U}(\xi) = \Lambda +U(\xi,a_1,a_1 \Lambda + a_2 )$ is another solution of (\ref{ex2eq3}). This concludes the proof. 
\end{proof}
Use of theorem \ref{th1} and (\ref{ex2eq7}) gives us the following solution of (\ref{ex2eq3}) in the form:
\begin{equation}\label{ex2eq19}
\bar{U}(\xi) = \Lambda + \frac{144 \left(a_1 \Lambda + a_2 \right)^2 c_{0,1} e^{\frac{\sqrt{-\left(a_1 \Lambda + a_2 \right)} \xi }{a_1}}}{\left(12 \left(a_1 \Lambda + a_2 \right) e^{\frac{\sqrt{-\left(a_1 \Lambda + a_2 \right)} \xi }{a_1}}-a_1 c_{0,1}\right)^2}.
\end{equation} 
 Now combination of  (\ref{ex2eq19}) and  (\ref{ex2eq2}) gives exact solution of  (\ref{ex2eq1}) in the form 
\begin{equation}\label{ex2eq20}
\bar{u}(x,y,t) =\Lambda + \frac{144 \ \left(a_1 \Lambda + a_2 \right)^2\ c_{0,1} \ \exp \left(\frac{\sqrt{-\left(a_1 \Lambda + a_2 \right)} }{a_1}\left(-\frac{1}{\sqrt{a_1}} \int \left\{a_1 \left( \alpha (t) + b\right)+c_2^2 \ \delta (t)-a_2\right\} \, dt+\sqrt{a_1} \ x+c_2 \ y-c_1\right)\right)}{\left[-a_1\ c_{0,1}+12\ \left(a_1 \Lambda + a_2 \right)\  \exp \left(\frac{\sqrt{-\left(a_1 \Lambda + a_2 \right)}}{a_1} \left(-\frac{1}{\sqrt{a_1}} \int \left\{a_1 \left( \alpha (t) + b\right)+c_2^2 \ \delta (t)-a_2\right\} \, dt+\sqrt{a_1} \ x+c_2 \ y-c_1\right)\right)\right]^2}.
\end{equation} 
Note that above solution is real and bounded for $a_1>0,\ \left(a_1 \Lambda + a_2 \right)<0, \ c_{0,1}>0.$ In the figure \ref{fig2}, profiles of solution (\ref{ex2eq20}) are displayed for arbitrary function and parameter values (a) $\alpha (t) =-25\ \cos (0.5 t^2),\ \delta (t)= 32\ \text{sech}(-0.5\  t),\ a_1 = 0.203,\ a_2 = -0.1,\ \Lambda= -0.5,\  c_1 = 0.5,\ c_2 = 0.1,\ c_{0,1} = 1,\ b =0.8$ for fixing $y =1$ in the first column and (b) $\alpha (t) =-8 \cos (1.5\ t+12),\ \delta (t)= -10 \text{sech}(1.4 t)+2.5\ t+1, \ a_1 = 0.34,\ a_2 = -0.12,\ \Lambda= -0.5,\ c_1 = 0.5,\ c_2 = 0.06,\ c_{0,1} = 1,\ b =0.8$  for fixing $x =1$ in the second column. It is clear from the figure that solution have snake like and periodic boomerang  like profile due to the presence of periodic values of $\alpha (t) ,\ \delta (t)$ in the solution. 
\begin{figure}
\includegraphics[scale=1.25]{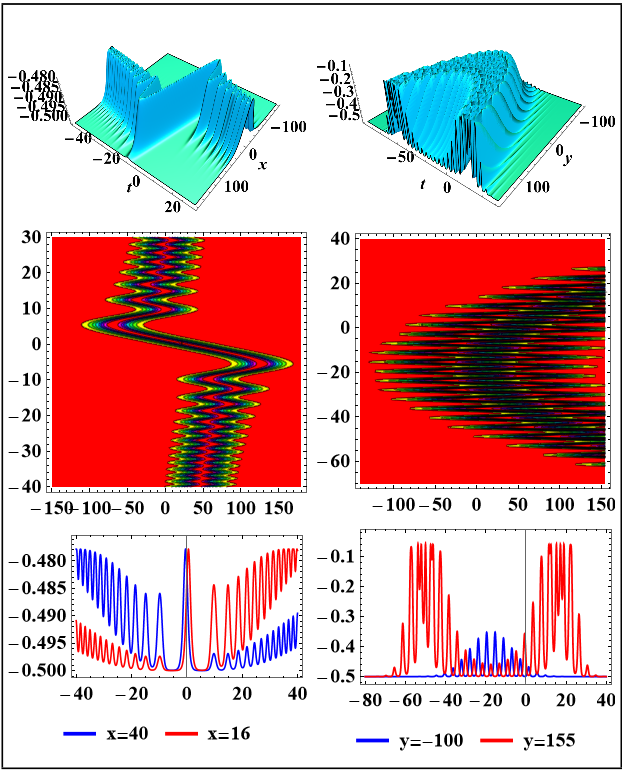}
\caption{3D, Contour and 2D Plots of solution (\ref{ex2eq20}) for arbitrary functions and parameter values (a) $\alpha (t) =-25\ \cos (0.5 t^2),\ \delta (t)= 32\ \text{sech}(-0.5\  t),\ y =1,\ a_1 = 0.203,\ a_2 = -0.1,\ \Lambda= -0.5,\  c_1 = 0.5,\ c_2 = 0.1,\ c_{0,1} = 1,\ b =0.8$ in shown first column and for (b) $\alpha (t) =-8 \cos (1.5\ t+12),\ \delta (t)= -10 \text{sech}(1.4 t)+2.5\ t+1,\ x =1, \ a_1 = 0.34,\ a_2 = -0.12,\ \Lambda= -0.5,\ c_1 = 0.5,\ c_2 = 0.06,\ c_{0,1} = 1,\ b =0.8$ those are showed in the 2nd column. }\label{fig2}
\end{figure}
 We can also combine (\ref{ex2eq15}) with theorem \ref{th1} to get another exact solution of  (\ref{ex2eq1}) like (\ref{ex2eq20}), which only differ in sign of variable $\xi$ and integration constant.  It is important to note that based on parameter restrictions of these derived solutions can be recast in several hyperbolic and trigonometric presented above and in the several forms presented in \citep{das2020chirped}.  For simplicity we have skipped those calculations here.
 
\textbf{Case-3.\  For initial Condition $V(0)= v_0,\ V'(0)= 0,\ V''(0)= v_1,\ V'''(0)= 0.$ }\\
To address (\ref{ex2eq18}) given the specified initial conditions, we examine the linear and nonlinear components in the structure

$\hat{{\cal L}}[V](\xi)= V^{(4)}(\xi), \ \  {\cal N}[V](\xi)=-\frac{a_{2}}{a_{1}^{2}} \left(V^{(2)} \! \left(\xi \right)\right)-\frac{1}{a_{1}} \left\{ \left(V' \! \left(\xi \right)\right)^{2}+\left(\Lambda +V \! \left(\xi \right)\right) \left(V^{(2)} \! \left(\xi \right)\right) \right\}.$

Subsequent to the section \ref{red}, we can simplify the equation (\ref{ex2eq18}) to the linear system of equations
\begin{equation}\label{case3a1}
  \hat{\cal L}\left[V_{0}(\xi) \right]=0, \ \ 
  \hat{\cal L}\left[V_{k}(\xi) \right] = {\cal A}_{k-1}(x),\ \ k \geq 1.
\end{equation}
Here the fundamental set of solutions of the linear part $\hat{\cal L}\left[V(\xi) \right] \sim   V^{(4)}(\xi )= 0$ is given by \\
 $\left\{ \tilde{V}_1(\xi)= \xi^3,\ \tilde{V}_2(\xi)= \xi^2, \ \tilde{V}_3(\xi)=\xi ,\ \tilde{V}_4(\xi)=1 \right\}$.  Thus, the solution for the linear part of the first equation in (\ref{case3a1}) can be expressed as $V_0(\xi)= c_3\ \xi^3+c_2\ \xi^2+c_1 \ \xi+c_0.$ Applying the initial conditions $V_0(0)= v_0,\ V_0'(0)= 0,\ V_0''(0)= v_1,\ V_0'''(0)= 0$ yields $c_3= 0,\ c_2= v_1,\ c_1= 0,\ c_0= v_0.$ Therefore, the system of equations in (\ref{case3a1}) can be addressed through MMVP utilizing the iterative scheme outlined in (\ref{eq2p13}) in the form
 \begin {eqnarray}\label{case3a2}
&& V_0(\xi) = v_1\ \xi^2+v_0, \nonumber \\
&& V_k(\xi) = \sum_{j=1}^{n}(-1)^{n-j}    \tilde{V}_{j}(\xi) \int \frac{{\cal A}_{k-1}(\xi) \ W_j(\xi)}{P_0(\xi)\ W(\xi)} d \xi,   \ \ k \geq 1.
\end{eqnarray}
The iterative method outlined above provides us with these correction terms
\begin {eqnarray}\label{case3a3}
&& V_0(\xi)=v_1\ \xi^2+v_0,\ V_1(\xi)=-\left(\left(v_{1} \xi^{2}+5 \Lambda +5 v_{0}\right) a_{1}+5 a_{2}\right) v_{1} \xi^{4}/(60 a_{1}^{2}) \nonumber 
,\\ 
&& V_2(\xi)=\xi ^6 v_1 \left(a_1^2 \left(210 \Lambda ^2+135 \Lambda  \xi ^2 v_1+15 v_0 \left(28 \Lambda +9 \xi ^2 v_1\right)+14 \xi ^4 v_1^2+210
   v_0^2\right)\right. \nonumber  \\
   && \ \ \  \ \ \ \ \ \ \left.+15 a_2 a_1 \left(28 \Lambda +9 \xi ^2 v_1+28 v_0\right)+210 a_2^2\right) /(75600 a_1^4) \nonumber 
,\\ 
   && V_3(\xi)=-\xi ^8 v_1 \left(13 a_2 a_1^2 \left(1485 \Lambda ^2+1782 \Lambda  \xi ^2 v_1+594 v_0 \left(5 \Lambda +3 \xi ^2 v_1\right)+254 \xi ^4
   v_1^2+1485 v_0^2\right)\right. \nonumber  \\
   && \ \ \  \ \ \ \ \ \ \left.+a_1^3 \left(6435 \Lambda ^3+11583 \Lambda ^2 \xi ^2 v_1+13 v_0 \left(1485 \Lambda ^2+1782 \Lambda  \xi ^2 v_1+254 \xi
   ^4 v_1^2\right)+3302 \Lambda  \xi ^4 v_1^2\right.\right. \nonumber \\
   && \ \ \  \ \ \ \ \ \ \left. \left.+3861 v_0^2 \left(5 \Lambda +3 \xi ^2 v_1\right)+231 \xi ^6 v_1^3+6435 v_0^3\right)+3861 a_2^2 a_1
   \left(5 \Lambda +3 \xi ^2 v_1+5 v_0\right)+6435 a_2^3\right)\nonumber \\ 
   &&  \ \ \  \ \ \ \ \ \  /(129729600 a_1^6), \nonumber \\ 
&& V_4(\xi)=\xi ^{10} v_1 \left(3060 a_2^2 a_1^2 \left(3003 \Lambda ^2+7644 \Lambda  \xi ^2 v_1+546 v_0 \left(11 \Lambda +14 \xi ^2 v_1\right)+1546
   \xi ^4 v_1^2+3003 v_0^2\right)\right. \nonumber \\
   && \ \ \  \ \ \ \ \ \ \left.+153 a_2 a_1^3 \left(40040 \Lambda ^3+152880 \Lambda ^2 \xi ^2 v_1+40 v_0 \left(3003 \Lambda ^2+7644 \Lambda 
   \xi ^2 v_1+1546 \xi ^4 v_1^2\right)+61840 \Lambda  \xi ^4 v_1^2\right.\right. \nonumber \\
   && \ \ \  \ \ \ \ \ \ \left. \left.+10920 v_0^2 \left(11 \Lambda +14 \xi ^2 v_1\right)+5481 \xi ^6 v_1^3+40040
   v_0^3\right)+a_1^4 \left(1531530 \Lambda ^4+7796880 \Lambda ^3 \xi ^2 v_1\right.\right.\nonumber  \\
   && \ \ \  \ \ \ \ \ \ \left. \left.+4730760 \Lambda ^2 \xi ^4 v_1^2+3060 v_0^2 \left(3003 \Lambda
   ^2+7644 \Lambda  \xi ^2 v_1+1546 \xi ^4 v_1^2\right)+153 v_0 \left(40040 \Lambda ^3 
   \right.\right.\right. \nonumber \\
   && \ \ \  \ \ \ \ \ \ \left. \left. \left.
   +152880 \Lambda ^2 \xi ^2 v_1+61840 \Lambda  \xi ^4
   v_1^2+5481 \xi ^6 v_1^3\right)+838593 \Lambda  \xi ^6 v_1^3+556920 v_0^3 \left(11 \Lambda +14 \xi ^2 v_1\right)\right.\right. \nonumber \\
   && \ \ \  \ \ \ \ \ \ \left. \left. +44198 \xi ^8 v_1^4+1531530
   v_0^4\right)+556920 a_2^3 a_1 \left(11 \Lambda +14 \xi ^2 v_1+11 v_0\right)+1531530 a_2^4\right) \nonumber  \\
   && \ \ \  \ \ \ \ \ \ \ /(2778808032000 a_1^8),\nonumber  \\
   && \ \ \  \ \ \ \ \ \ \  \cdots 
\end{eqnarray}
Through symbolic computation and manipulation, the closed form of the sum for the above  iterative terms is given by the generating function
\begin{eqnarray}\label{case3a4}
V(\xi) =v_0 +\frac{-3 a_1 \left(\Lambda +v_0\right)+\sqrt{9 \left(a_1 \left(\Lambda +v_0\right)+a_2\right)^2+48 a_1^3 v_1}-3 a_2}{2 a_1} \text{sn}^2\left[\chi \ \xi, \kappa  \right],
\end{eqnarray} 
where
\begin{eqnarray}\label{case3a5}
&& \chi = \sqrt{\frac{3 a_1 \Lambda +\sqrt{9 \left(a_1 \left(\Lambda +v_0\right)+a_2\right)^2+48 a_1^3 v_1}+3 a_1 v_0+3 a_2}{24 a_1^2}},\nonumber \\
\text{and} \ \ \ \ \ \ \ \ \ \ \ \ \ \ \ \ && \nonumber \\
&& \kappa = \sqrt{ \frac{\left(a_1 \left(\Lambda +v_0\right)+a_2\right) \sqrt{9 \left(a_1 \left(\Lambda +v_0\right)+a_2\right)^2+48 a_1^3 v_1}-3 \left(a_1
   \left(\Lambda +v_0\right)+a_2\right)^2}{8 a_1^3 v_1}-1}.
\end{eqnarray} 
Under the transformation
$v_{0} = 
-\frac{-4 k^{2} \vartheta^{2} a_{1}^{2}-4 \vartheta^{2} a_{1}^{2}+a_{1} \Lambda +a_{2}}{a_{1}}
, v_{1} = -12 \vartheta^{4} k^{2} a_{1}$ and transformation property $\mathrm{sn} (z; 1/k) = k\ \mathrm{sn} (z/k; k)$ of the Jacobi $\mathrm{sn}$ function \citep{frank2010nist} above solution reduces to  a simple form 
\begin{eqnarray}\label{case3a6}
V(\xi)=\frac{4 \vartheta^{2} \left(k^{2}+1\right) a_{1}^{2}-a_{1} \Lambda -a_{2}}{a_{1}}  -12  k^{2} \vartheta^{2} a_{1} \ \mathrm{sn}^{2}\! \left[\vartheta \ \xi, \ k \right].
\end{eqnarray} 
Utilizing the property of the Jacobi $\mathrm{sn}$ function $\mathrm{sn}^2 (z; k) + \mathrm{cn}^2 (z; k) = k^2 \mathrm{sn}^2 (z; k) + \mathrm{dn}^2 (z; k) = 1$, the above solution recast to following two elliptic forms
\begin{eqnarray}\label{case3a7}
V(\xi)=\frac{\left(-8 k^{2}+4\right) \vartheta^{2} a_{1}^{2}-a_{1} \Lambda -a_{2}}{a_{1}}+
 12  k^{2} \vartheta^{2} a_{1} \ \mathrm{cn}^{2}\! \left[\vartheta \  \xi, \ k \right],
\end{eqnarray}
and 
\begin{eqnarray}\label{case3a8}
V(\xi)=\frac{4 \vartheta^{2} \left(k^{2}-2\right) a_{1}^{2}-a_{1} \Lambda -a_{2}}{a_{1}}
+12  \vartheta^{2} a_{1} \ \mathrm{dn}^{2}\! \left[\vartheta \ \xi, \ k \right].
\end{eqnarray}
Currently, the combination of any one from (\ref{case3a4}), (\ref{case3a6})-(\ref{case3a8}) with (\ref{ex2eq17a}) and (\ref{ex2eq2}) results in an exact solution for (\ref{ex2eq1}). Here, for ease, we combine (\ref{case3a6}) with (\ref{ex2eq17a}) and (\ref{ex2eq2}), resulting in the exact solution of (\ref{ex2eq1}) expressed as
\begin{eqnarray}\label{case3a9}
u(x,y,t) =&&\frac{4 \vartheta^{2} \left(k^{2}+1\right) a_{1}^{2}-a_{1} \Lambda -a_{2}}{a_{1}}  \nonumber \\
&&-12  k^{2} \vartheta^{2} a_{1} \ \mathrm{sn}^{2}\! \left[\vartheta \ \left(-\frac{1}{\sqrt{a_1}} \int \left\{a_1 \left( \alpha (t) + b\right)+c_2^2 \ \delta (t)-a_2\right\} \, dt+\sqrt{a_1} \ x+c_2 \ y-c_1\right), \ k \right].
\end{eqnarray} 
It is important to note that the solution obtained above is valid for $a_1 > 0$.  In the figure \ref{fig3} profiles of solution (\ref{case3a9}) are displayed for arbitrary function and parameter values (a) $\alpha (t) =-5 \cos (0.6 t+12),\ \delta (t)= 2 \text{sech}(0.1 t)-0.1,\ y =1,\ \Lambda = 0.6,\vartheta = 0.25,k= 0.8,\ a_1 = 0.3,\ a_2 = -0.04,\ c_1 = 0.5,\ c_2 = 0.3,\  b =0.8$ for fixed $y =1$ in the first column and (b) $\alpha (t) =-5 \cos (5 t+12),\ \delta (t)= -5 \text{sech}(0.4 t)+7\ t+5, \ \Lambda = 0.6,\vartheta = 0.3,\ k= 0.8,\ a_1 = 0.8,\ a_2 = -0.5,\ c_1 = 0.5,\ c_2 = 0.3\ b =0.8$  for fixed $x =1$ in the second column. The figure clearly shows that the solution exhibits a snake-like and periodic boomerang-like profile, attributed to the periodic values of $\alpha (t)$ and $\delta (t)$ present in the solution.
\begin{figure}
\includegraphics[scale=1.25]{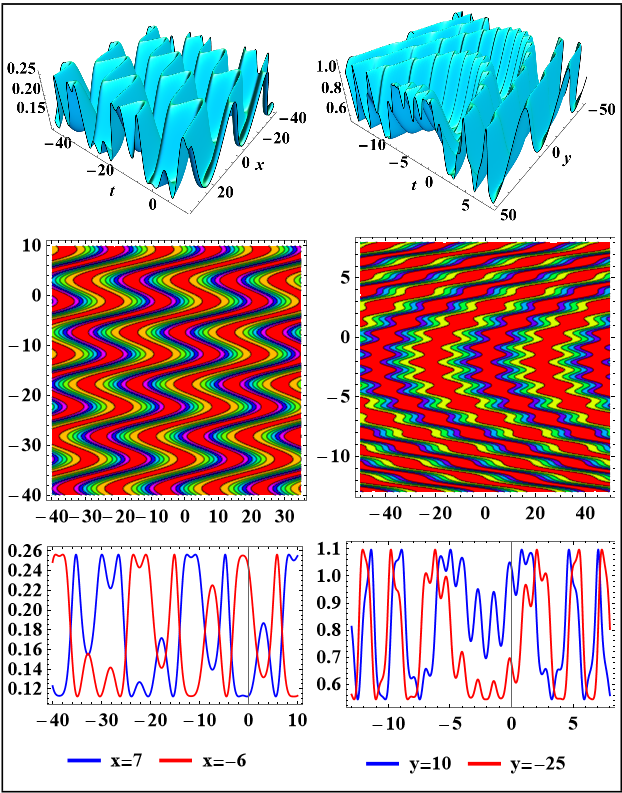}
\caption{3D, Contour and 2D Plots of solution (\ref{case3a9}) for arbitrary functions and parameter values (a) $\alpha (t) =-5 \cos (0.6 t+12),\ \delta (t)= 2 \text{sech}(0.1 t)-0.1,\ y =1,\ \Lambda = 0.6,\vartheta = 0.25,k= 0.8,\ a_1 = 0.3,\ a_2 = -0.04,\ c_1 = 0.5,\ c_2 = 0.3,\  b =0.8$ in shown first column and for (b) $\alpha (t) =-5 \cos (5 t+12),\ \delta (t)= -5 \text{sech}(0.4 t)+7\ t+5, \ x=1, \ \Lambda = 0.6,\vartheta = 0.3,k= 0.8,\ a_1 = 0.8,\ a_2 = -0.5,\ c_1 = 0.5,\ c_2 = 0.3\ b =0.8$  those are showed in the 2nd column. }\label{fig3}
\end{figure}
\subsection{Advantages and Limitations of MMVP}\label{AdLim}
The MMVP offers the following benefits as a means of delivering a exact solution:
1. Nonlinear ordinary differential equations (NLODEs) of any order with several linear and nonlinear terms can be solved using it.
2. For nonlinear differential equations that are taken into consideration, it most of cases gives the closed form solution of the acquired series solution.
3. It can provide closed form solutions of NLODEs in terms of hyperbolic, trigonometric, algebraic, Jacobi elliptic functions.
4. Depending on the order of the differential equations, the linear terms that are adjusted, and the initial and boundary conditions that are applied, the method can yield numerous specific exact solutions.\\
In addition to its benefits, this approach has the following drawbacks:
1. When the complementary functions of the linear component are unknown, the approach is unable to solve NLODEs. 
2. For considered NLODEs, the closed form solution may not always be found from the provided series solution.
\section{Conclusion}\label{con}
We have presented a novel approach to derive the exact solution of nonlinear ordinary differential equations in this work.  Here, we develop exact solutions for higher-order nonlinear boundary and initial value problems for the first time by combining a linearization methodology for nonlinear equations with method of variation of parameters and Adomian polynomials. When it comes to finding exact solutions for nonlinear equations, the proposed approach is very successful. We used the novel approach to the extended to (2 + 1) dimensions thermophoretic motion equation, which was obtained from the motions of wrinkle waves in substrate-supported graphene sheets. It is demonstrated that this approach may effectively produce closed form solutions in terms of  trigonometric, exponential, hypergeometric, algebraic, and Jacobi elliptic functions, respectively. Three figures are used to present the profiles of three solutions, each of which has some intriguing characteristics.  While previous MVP applications in the literature deal with approximate solutions, in this work we introduced modified  MVP using the Adomian polynomial and applied it to achieve exact closed form solutions to NLODEs for the first time.  Without the need for perturbation, discretizations, or the identification of Lagrange multipliers, the suggested iterative method finds the exact closed form solutions. The suggested approach delivers the solution through a quickly converging series that easily converges to a closed form, enhancing the method's efficiency in obtaining exact closed-form solutions. These findings are distinctive and not found in existing literature. 
\section*{Acknowledgement}
The authors thank the article's referees for their important ideas to improve the article's content. 
\section*{References}
\bibliographystyle{plain}
\bibliography{cNLSEref}
\end{document}